\documentclass{amsart}
\usepackage{amsmath,amssymb,graphicx,setspace,verbatim}
\usepackage{color, tikz, url}
\input xy
\xyoption{all}
\newtheorem{prop}{Proposition}[section]
\newtheorem{coro}[prop]{Corollary}

\newtheorem{example}{Example}
\newtheorem{remark}{Remark}
\newtheorem{thm}[prop]{Theorem}
\newtheorem{lemma}[prop]{Lemma}

\newtheorem{construction}[prop]{Construction}

\begin{document}
\title{On residual connectedness in chiral geometries}

\author[D. Leemans]{Dimitri Leemans}
\address{Dimitri Leemans, Universit\'e Libre de Bruxelles, D\'epartement de Math\'ematique, C.P.216 - Alg\`ebre et Combinatoire, Boulevard du Triomphe, 1050 Brussels, Belgium
}
\email{dleemans@ulb.ac.be}

\author[P. Tranchida]{Philippe Tranchida}
\address{Philippe Tranchida, Department of Mathematical Sciences, KAIST, 291 Daehak-ro Yuseong-gu, Daejeon, 34141, South Korea}
\email{ptranchi@kaist.ac.kr}

\date{}
\maketitle
\begin{abstract}
We show that a chiral coset geometry constructed from a $C^+$-group necessarily satisfies residual connectedness and is therefore a hypertope.
\end{abstract}
\textbf{Keywords}: coset geometries, hypertopes, chirality, $C^+$-groups, residual connectedness.

\section{Introduction}

The concept of a hypertope, introduced recently in \cite{Hypertopes}, is a generalization of an abstract polytope. There are different but equivalent ways to define (abstract) polytopes, one of which being that its faces form a partially ordered set that is a thin, residually connected geometry. This has been generalized in \cite{Hypertopes} to include structures built from a set of (what we still call) faces that do not form a partially ordered set.

Several papers have been written on the subject (see for instance~\cite{tetrahyp,FL2018,FLPW,FLW,ens,hou}) and in those paper dealing with chiral hypertopes, the check for residual connectedness has been a bit of a struggle.

In the regular case, there is an easy way to test if a coset geometry is residually connected (see Theorem~\ref{theorem46}). The aim of this paper is to prove a similar result for chiral geometries.

The paper is organized as follows.
In Section~\ref{prelim}, we recall the definitions and notation needed to understand this paper.
In Section~\ref{hypertopes}, we give the basic definitions on hypertopes and recall how to construct regular and chiral hypertopes from $C$-groups and $C^+$-groups.
In Section~\ref{last}, we prove our main result, namely:
\begin{thm}\label{main}
Let $G$ be a group and $R$ be a set of generators of $G$ such that \((G^+,R)\) is a \(C^+\)-group.
Let \(\Gamma= \Gamma(G^+,(G_i^+)_{i\in I})\) be the coset geometry associated to \(G^+\). Then, if \(\Gamma\) is chiral, it is residually connected.
\end{thm}

This theorem shows that residual connectedness for hypertopes constructed from $C^+$-groups follows from chirality in exactly the same way as it follows from flag-transitivity for hypertopes constructed from C-groups.
In other words, it is enough to test that a coset geometry constructed from a $C^+$-group is a chiral geometry to be able to conclude that it is a hypertope.

\section{Preliminaries}\label{prelim}

\subsection{Incidence geometries}
An {\it incidence system}  is a 4-tuple $\Gamma := (X, *, t, I)$ such that
\begin{itemize}
\item $X$ is a set whose elements are called the {\it elements}, or {\it faces}, of $\Gamma$;
\item $I$ is a set whose elements are called the {\it types} of $\Gamma$;
\item $t:X\rightarrow I$ is a {\it type function}, associating to each element $x\in X$ of $\Gamma$ a type $t(x)\in I$;
\item $*$ is a binary relation on $X$ called {\em incidence}, that is reflexive, symmetric and such that for every $x,y\in X$, if $x*y$ and $t(x) = t(y)$ then $x=y$.
\end{itemize}
The {\it rank} of $\Gamma$ is the cardinality of $I$. 
A {\it flag} is a set of pairwise incident elements of $\Gamma$.
The {\it type} of a flag $F$ is $\{t(x) : x \in F\}$ and a flag of type $I$ is called a {\em chamber}.

An incidence system $\Gamma$ is a {\it geometry} or {\it incidence geometry} if every flag of $\Gamma$ is contained in a chamber.
An element $x$ is {\em incident} to a flag $F$, and we write $x*F$ for that, provided $x$ is incident to all elements of $F$.
If $\Gamma = (X, *, t, I)$ is an {\it incidence geometry} and $F$ is a flag of $\Gamma$,
the {\em residue} of $F$ in $\Gamma$ is the incidence geometry $\Gamma_F := (X_F, *_F, t_F, I_F)$ where $X_F := \{ x \in X : x * F, x \not\in F\}$;  $I_F := I \setminus t(F)$;  $t_F$ and $*_F$ are the restrictions of $t$ and $*$ to $X_F$ and $I_F$.

The {\it incidence graph} of $\Gamma$ is the graph whose vertex set is $X$ and where two distinct vertices are joined provided the corresponding elements of $\Gamma$ are incident.

An incidence system $\Gamma$ is {\em connected} if its incidence graph is connected. It is
{\em residually connected} when each residue of rank at least two of $\Gamma$ (including $\Gamma$ itself) has a connected incidence graph. 

An incidence system $\Gamma$ is {\it thin} (respectively {\em firm}) when every residue of rank one of $\Gamma$ contains exactly (respectively at least) two elements. 

Let $\Gamma =(X,*, t,I)$ be an incidence system.
An {\em automorphism} of $\Gamma$ is a 
permutation $\alpha$ of $X$ inducing a permutation of $I$ such that
\begin{itemize}
\item for each $x$, $y\in X$, $x*y$ if and only if $\alpha(x)*\alpha(y)$;
\item for each $x$, $y\in X$, $t(x)=t(y)$ if and only if $t(\alpha(x))=t(\alpha(y))$.
\end{itemize}
An automorphism $\alpha$ of $\Gamma$ is called {\it type preserving} when for each $x\in X$, $t(\alpha(x))=t(x)$.
The set of type-preserving automorphisms of $\Gamma$ is a group denoted by $Aut_I(\Gamma)$.
The set of automorphisms of $\Gamma$ is a group denoted by  $Aut(\Gamma)$.
A group $G\leq Aut_I(\Gamma)$ acts {\em flag-transitively} on $\Gamma$ if $G$ is transitive on the set of chambers of $\Gamma$. In this case, we also say that $\Gamma$ is {\em flag-transitive}.
The following proposition shows how, starting from a group $G$, we can construct an incidence system whose type-preserving automorphism group contains $G$.

\begin{prop}(Tits, 1956)~\cite{Tits}\label{tits}
Let $n$ be a positive integer
and $I:= \{0,\ldots ,n-1\}$.
Let $G$ be a group together with a family of subgroups ($G_i$)$_{i \in I}$, $X$ the set consisting of all cosets $G_ig$ with $g \in G$ and $i \in I$, and $t : X \rightarrow I$ defined by $t(G_ig) = i$.
Define an incidence relation $*$ on $X\times X$ by:
\begin{center}
$G_ig_1 * G_jg_2$ if and only if $G_ig_1 \cap G_jg_2 \neq \emptyset$.
\end{center}
Then the 4-tuple $\Gamma := (X, *, t, I)$ is an incidence system having a chamber.
Moreover, the group $G$ acts by right multiplication on $\Gamma$ as a group of type preserving automorphisms.
Finally, the group $G$ is transitive on the flags of rank less than 3.
\end{prop}

The incidence system constructed by the proposition above will be denoted by $\Gamma(G; (G_i)_{i\in I})$ and might not be a geometry, but if it is a geometry we call it a {\em coset geometry}.

Given a family of subgroups $(G_i)_{i\in I}$ we define $G_J := \cap_{j \in J}G_j$. The subgroups $G_J$ are called the {\em parabolic subgroups} of the coset geometry $\Gamma(G; (G_i)_{i\in I})$. Moreover $G_j = G_{\{j\}}$ for each $j\in I$ and the subgroups $G_j$ are called the {\em maximal parabolic subgroups} of $\Gamma$.

In the case of flag-transitive coset geometries, there is an easy group-theoretical way to test residual connectedness, which was originally proved by Francis Buekenhout and Michel Hermand.

 \begin{thm}\cite[Corollary 1.8.13]{Diagram Geometry}
\label{proposition:RC}
Suppose \(I\) is finite and let \(\Gamma= \Gamma(G,(G_i)_{i\in I})\) be a geometry over \(I\) on which \(G\) acts flag transitively. Then \(\Gamma\) is residually connected if and only if \(G_J = \langle G_{J\cup\{i\}}~|~i\in I\setminus J\rangle\) for every \(J\subseteq I\) with \(|I\setminus J| \ge 2\).
\end{thm}

A geometry $\Gamma$ is {\em chiral} if $Aut_I(\Gamma)$ has two orbits on the chambers such that any two adjacent chambers lie in distinct orbits.
 
Observe that if a geometry is chiral, it is necessarily thin as if a rank one residue contains more than two elements, this contradicts chirality.
 
\section{Hypertopes}\label{hypertopes}
Most of the definitions in this section come from~\cite{Hypertopes}.
 A \emph{hypertope} is a thin, residually connected geometry. 
 A hypertope $\Gamma$ is {\em regular} if $\Gamma$ is a flag-transitive geometry.
 A hypertope $\Gamma$ is {\em chiral} if $\Gamma$ is a chiral geometry.
 
 Let $\Gamma(X,*,t,I)$ be a thin geometry and $i\in I$. If $C$ is a chamber of $\Gamma$, we let $C_i$ denote the chamber {\em $i$}-adjacent to $C$, that  is, the chamber that differs from $C$ only in its $i$-face. 
 
\subsection{C-groups and regular hypertopes}

Given a regular hypertope $\Gamma$ and a chamber $C$ of $\Gamma$, for each $i\in I$ let $\rho_i$ denote the automorphism mapping $C$ to $C_i$.
Then $\{\rho_0,\ldots, \rho_{n-1}\}$ is a generating set for $Aut_I(\Gamma)$ and $G_i=\langle \rho_j\,|\, j\neq i\rangle$ is the stabilizer of the $i$-face of $C$. 
Moreover $(Aut_I(\Gamma),\{ \rho_0,\ldots, \rho_{n-1}\})$ is a \emph{C-group}  \cite[Theorem~4.1]{Hypertopes}, that is, $\{ \rho_0,\ldots, \rho_{n-1}\}$ is a set of involutions generating $Aut_I(\Gamma)$ and satisfying the following condition, called the \emph{intersection condition}.
$$\forall I, J \subseteq \{0, \ldots, n-1\},
\langle \rho_i \mid i \in I\rangle \cap \langle \rho_j \mid j \in J\rangle = \langle \rho_k \mid k \in I \cap J\rangle.$$

From a C-group we can get a hypertope when the incidence system arising from Proposition~\ref{tits} is flag-transitive, as shown in the following theorem.

\begin{thm}\cite[Theorem 4.6]{Hypertopes}\label{theorem46}
Let $G=\langle \rho_0, \ldots, \rho_{n-1}\rangle$ be a C-group of rank $n$ and let $\Gamma := \Gamma(G;(G_i)_{i\in I})$ with $G_i := \langle \rho_j | j \in I\setminus \{i\} \rangle$ for all $i\in I:=\{0, \ldots, n-1\}$.
If $G$ is flag-transitive on $\Gamma$, then $\Gamma$ is a regular hypertope.
\end{thm}

Observe that Theorem~\ref{main} is the equivalent of the latter theorem for chiral geometries.

\subsection{C$^+$-groups}\label{section6}

We now consider another class of groups 
from which we will be able to construct hypertopes. These hypertopes may or may not be regular. In the latter case, they will be chiral. 

Consider a pair $(G^+,R)$ with $G^+$ being a group and $R:=\{\alpha_1, \ldots, \alpha_{r-1}\}$ a set of generators of $G^+$.
Define $\alpha_0:=1_{G^+}$
 and $\alpha_{ij} := \alpha_i^{-1}\alpha_j$ for all $0\leq i,j \leq r-1$.
 {\color{black}Observe that $\alpha_{ji} = \alpha_{ij}^{-1}$.}
Let $G^+_{\color{black}J} := \langle \alpha_{ij} \mid i,j \in {\color{black}J}\rangle$ for ${\color{black}J}\subseteq \{0, \ldots, r-1\}$.

If the pair $(G^+,R)$ satisfies condition (\ref{IC+}) below called the {\em {\color{black}intersection condition}} {\color{black}IC}$^+$, we say that $(G^+,R)$ is a {\em $C^+$-group}.
\begin{equation}\label{IC+}
\forall {\color{black}J, K} \subseteq \{0, \ldots, r-1\}, with \;{\color{black}|J|, |K|} \geq 2,
G^+_{\color{black}J} \cap G^+_{\color{black}K} = G^+_{{\color{black}J\cap K}}.
\end{equation}
It follows immediately from the {\color{black}intersection condition} {\color{black}IC}$^+$, that $R$ is an independent generating set for $G^+$, that means that $\alpha_i \not\in \langle \alpha_j : j \neq i\rangle$.

We now explain how to construct a coset geometry from a group and an independent generating set of this group.

\begin{construction}\label{hyper}
Let $I=\{1,\ldots, r-1\}$,  $G^+$ be a group and {\color{black}$R:=\{\alpha_1,\ldots, \alpha_{r-1}\}$} be an independent generating set of $G^+$.
Define $G^+_i := \langle \alpha_j | j \neq i \rangle$ for $i=1, \ldots, r-1$ and $G^+_0 := \langle \alpha_1^{-1}\alpha_j | j \geq 2 \rangle$.
The coset geometry $\Gamma(G^+,R) := \Gamma(G^+; (G^+_i)_{i\in \{0,\ldots,r-1\}})$ constructed using Tits' algorithm (see Proposition~\ref{tits}) is the geometry associated to the pair $(G^+,R)$. 
\end{construction}

The coset geometry $\Gamma(G^+,R)$ gives an incidence system using Tits algorithm. 
If this incidence system is a chiral hypertope, then $(G^+,R)$ is necessarily a $C^+$-group by the following theorem.

\begin{thm}\cite[Theorem 7.1]{Hypertopes}\label{cplusgroup}
Let $I:=\{0, \ldots, r-1\}$ and let $\Gamma$ be a chiral hypertope of rank $r$. Let $C$ be a chamber of $\Gamma$.
The pair $(G^+,R)$ where $G^+=Aut_I(\Gamma)$ and $R$ is the set of distinguished generators of $G^+$ with respect to $C$ is a $C^+$-group.
\end{thm}

\begin{coro}\cite[Corollary 7.2]{Hypertopes}
The set $R$ of Theorem~\ref{cplusgroup} is an independent generating set for $G^+$. 
\end{coro}

\section{Residual connectedness of chiral hypertopes}
\label{last}

As we saw in the previous section, if a coset geometry $\Gamma =\Gamma(G,(G_i)_{i\in \{1,\ldots,n\}})$ is flag-transitive, it is fairly easy to verify if it is residually connected. Not only do we know that it is sufficient to check every residue of the base chamber \((G_0,...,G_n)\), Theorem~\ref{proposition:RC} even gives us a criterion for residual connectedness based only on the maximal parabolic subgroups \((G_i)\)'s and their intersections. The goal of this section is to prove Theorem~\ref{main} which gives a similar result for chiral hypertopes.

The fact that is it sufficient to check the residues contained in a base chamber is actually pretty straightforward as the following theorem shows.

\begin{thm}
 Let \(\Gamma\) be a chiral geometry and let \( C\) be a chamber of \(\Gamma\). Then \(\Gamma\) is residually connected if and only if \(\Gamma_{F}\) is connected  for all flags F of C of corank at least $2$.
\end{thm}
\begin{proof}
 If \(\Gamma\) is residually connected, then, obviously, \(\Gamma_F\) is connected for all flags F \(\in \Gamma\) of corank at least $2$.

Now, let F be a flag of \(\Gamma\) of corank at least 2 that is not necessarily included in $C$ and let us prove that \(\Gamma_F\) is connected. The flag F is contained in a chamber \(C_{1}\) of \(\Gamma\) since \(\Gamma\) is a geometry. If \(C_{1}\) is in the same orbit as \(C\), then there exists an \(\alpha \in Aut_{I}(\Gamma)\) sending \(C_{1}\) to \(C\). Then, this \(\alpha\) sends F on a flag \(\alpha(F)\) of \(C\) and \(\Gamma_{F}\) is isomorphic to \(\Gamma_{\alpha(F)}\) which is connected by hypothesis.
Suppose now \(C_{1}\) is not in the same orbit as \(C\), and take a type \(i\in I\) such that \(i \notin J\). Then \(C_{1}\) is in the same orbit as \(C^i\), the \(i\)-adjacent chamber to \(C\).
Then again, we can find \(\alpha \in Aut_I(\Gamma)\) sending \(C_1\) to \(C^i\). Now notice that, since \(C^i\) coincides with \(C\) except for the element of type \(i\), \(\alpha(F)\) actually maps F onto a flag \(\alpha(F) \in C\). By the same reasoning as before, we obtain that \(\Gamma_F\) is then connected.
We have proved that \(\Gamma_F\) is connected for all flags F \(\in\Gamma\) with corank at least 2, thus \( \Gamma\) is residually connected.
\end{proof}
The proof of the above theorem suggests the following result.
\begin{prop}\label{jnoti}
Let \(\Gamma\) be a chiral geometry over \(I = \{1,...,n\}\). Then \(Aut_I(\Gamma)\) is transitive on the set of flags of type \(J\) for any \(J\subseteq I\) such that \(|J| < |I|\). 
\end{prop}
\begin{proof}
Take two flags \(F\) and \(F'\) of type \(J\). Then \(F\) is contained in at least two chambers \(C_1\) and \(C_2\) lying in different orbit. Take \(C'\) a chamber of \(\Gamma\) containing \(F'\).
Then either \(C_1\) and \(C'\) or \(C_2\) and \(C'\) are in the same orbit. In any case, there exist an element \(\rho \in Aut_I(\Gamma)\) such that \(\rho(C_i) = C'\) for \( i = 1\) or \(2\) and thus \(\rho(F) = F'\).
\end{proof}

In other words, \(Aut_I(\Gamma)\) is transitive on every set of flags of a given type, except on chambers, since there are exactly \(2\) orbits for those.

We just proved that we only need to check that all the residues of rank at least two of a given chamber are connected for a chiral geometry to be residually connected. This concludes the first part of our analogy with the regular case. 

We now focus on coset geometries that are chiral and try to find conditions on the maximal parabolic subgroups for these geometries to be residually connected.

Let \((G^+,R)\) be a \(C^+\)-group with set of generators \(R =\{\alpha_1,\alpha_2, ...,\alpha_{n-1}\}\).
We use  Construction~\ref{hyper} to create the incidence geometry \(\Gamma = (G^+, \{G_0^+, G_1^+, ..., G_{n-1}^+\})\) from \(G^+\).
We would like to find a condition for the residual connectedness of \(\Gamma\) which would only involve subgroups of \(G^+\), in analogy to the condition existing for flag-transitive coset geometries.

The main problem lies in the fact that little to nothing is known, in the chiral case, on how to express the residues of \(\Gamma\) in terms of coset geometries derived from \(G^+\)  while everything works perfectly fine in the regular case as we previously saw.

Let us recall here a Lemma from~\cite{Diagram Geometry}. 
\begin{lemma}[\cite{Diagram Geometry}, Lemma 1.8.9]\label{lemma1.8.9}
Let \(\Gamma = \Gamma(G, (G_i)_{i \in I})\) be the coset incidence system of \(G\) over \( (G_i)_{i \in I}\).
Then, for each \(J \subseteq I\), there is a natural injective homomorphism of incidence systems over \(I\setminus J\) 
\begin{center}
\( \varphi_J :  \Gamma(G_J,(G_{J\cup \{i\}})_{i\in I}) \mapsto \Gamma_{G_j | j\in J} \)
\end{center}
given by \(\varphi(aG_{J \cup \{i\}  }) = aG_i\) \((a \in G_J, i\in I\setminus J)\).
Furthermore, given \(J \subseteq I\), the homomorphism \(\varphi_J\) is surjective if and only if, for all \(i \in I\setminus J \), we have \( \bigcap_{j \in J} (G_j G_i) = G_J G_i \) and if \(\varphi_J\) is surjective for all \(J \subseteq I\), then \(\varphi_J\) is an isomorphism for all \(J\subseteq I \).
\end{lemma}

\begin{remark}
In particular, for \(J = \{j\}\) for some \( j\in I\), we have that \(\varphi_{\{j\}}\) is always a bijective homomorphism. This means, that for geometries of rank three, we always have a bijective homomorphism for residues of rank two. Here, it is important to note that a bijective homomorphism does not need to be an isomorphism. It is an isomorphism if and only if its inverse is also a bijective homomorphism.
\end{remark}

In the regular case, all those \(\varphi_j\)'s are isomorphisms and we therefore know that every residue of type \(J\) is isomorphic to \((G_J, (G_{J\cup \{i\}})_{i\in I\setminus J})\), which gives us a perfect description of the residues. We also know that, unfortunately, this cannot be the case for chiral geometries, since it can be easily proven that if every \(\varphi_J\) is an isomorphism, then \(G\) has to act flag-transitively on \(\Gamma\). 
Nonetheless, almost all of those \(\varphi_J\) are still surjective, and thus bijective, in the chiral case, as shown in the following result.

\begin{lemma}
\label{lemma:varphi}
Let \(\Gamma= \Gamma(G,(G_i)_{i\in I})\) be a coset incidence geometry over \(I\) and fix \(J\subseteq I\setminus \{i\}\). Then, \(G\) is transitive on the set of flag of type \(\{i\}\cup J\) if and only if \(G\) is transitive on the set of flags of type \(J\) and \(\bigcap_{j\in J}G_i  G_j=G_i G_J\).
\end{lemma}
\begin{proof}
We set \(J = \{j_1,...,j_k\}\) for some natural number \(k\).
Suppose first that \(G\) is transitive on the set of flags of type \(\{i\}\cup J\). Then \(G\) is obviously also transitive on the set of flags of type \(J\). Let us now take an element \(x \in \bigcap_{j\in J}G_i  G_j\). We will prove that \(x\) is also in \( G_i G_J\).
Since \(x \in \bigcap_{j\in J}G_i  G_j\), we have that \(x\) is in every \(G_iG_j\) and that therefore \(x^{-1}G_i \cap G_j \ne \emptyset\) for every \(j \in J\). This means that \(\{x^{-1}G_i,G_{j_1}, ...., G_{j_k}\}\) is a flag of \(\Gamma\). By transitivity hypothesis, we can thus find an element \(g\in G\) such that \(g(G_i) = x^{-1}G_i\) and \(G_{j_m} = G_{j_m}\) for every \(m = 1,...,k\). This means that \(x^{-1} \in gG_i\) and \(g\in G_J\). Therefore, we have \(x\in G_ig^{-1} \subseteq G_iG_J\) and this shows that \(\bigcap_{j\in J}G_i  G_j\subseteq G_i G_J\).

The other inclusion is trivial since any \(x \in G_J\) also belongs to every \(G_j\) and thus any \(y \in G_iG_J\) belongs to every \(G_iG_j\) for every \(j\in J\).
This concludes the proof of the only if part.

Suppose now that \(\bigcap_{j\in J}G_i  G_j=G_i G_J\) and that \(G\) is transitive on the set of flags of type \(J\) and let us take a flag \(F\) of type \(\{i\}\cup J\) in \(\Gamma\). We want to show that \(F\) lies in the same \(G\)-orbit as \(\{G_i,G_{j_1}, ...., G_{j_k}\}\).
Since \(G\) is transitive on the set of flags of type \(J\), there exists an element \(g\in G\) such that \(g(\{ x_1G_{j_1},...,x_kG_{j_k}\}) =\{G_{j_1},...,G_{j_k}\}\). Then \(g(X) = \{gx_0G_i,G_{j_1},...,G_{j_k}\}\). 
We can therefore suppose that \(F = \{xG_i,G_{j_1}, ...., G_{j_k}\}\) without loss of generality. 
By looking at the incidence relations in \(F\), since \(\bigcap_{j\in J}G_i  G_j=G_i G_J\) by hypothesis, we have that $xG_i * G_{j_l}$ for each $l=1,\ldots , k$ if and only if $xG_i$ meets all the other cosets on $G_J$. Therefore there exists an element $z \in \cap_{j\in J}G_j\cap xG_i$ such that $z^{-1}(F) = \{G_i,G_{j_1},...,G_{j_k}\}$.
This concludes the proof.
\end{proof}
We are now ready to state the main proposition of this section.
\begin{prop}
\label{chiral_final}
Let \(\Gamma= \Gamma(G,(G_i)_{i\in I})\) be a coset geometry over \(I\) and fix \(J\subseteq I\setminus \{i\}\). Then the following are equivalent:
\begin{enumerate}
\item[(1)] $G$ is transitive on the set of flags of type \(J \cup \{i\}\) for every \(i \in I\setminus J\).
\item[(2)] The homomorphism \(\varphi_J\) is surjective.
\item[(3)] For every \(i \in I \setminus J\), we have \(\bigcap_{j\in J} G_jG_i = G_JG_i\) and \(G\) is transitive on the set of flags of type \(J\).
\end{enumerate}
\end{prop}
\begin{proof}
The equivalence between \((1)\) and \((3)\) is given by Lemma~\ref{lemma:varphi} while the equivalence between \((2)\) and \((3)\) is given by Lemma~\ref{lemma1.8.9}.
\end{proof}
We thus see that, if \(\Gamma = \Gamma(G,(G_i)_{i\in I})\) is a chiral geometry, by Proposition~\ref{jnoti},  the group $G$ is transitive on every set of flags of type \(J \ne I\) and therefore, every \(\varphi_J\) is surjective, except those related to residues of rank one, which are of no importance for residual connectedness. Since the \(\varphi_J\)'s are surjective, we deduce a sufficient condition for the residue \(\Gamma_J\) to be connected.
\begin{prop}
Let \(\Gamma = \Gamma(G,(G_i)_{i\in I})\) be a chiral coset geometry over \(I\) and take a subset \(J\subseteq I\) such that \(|J| < |I|\). If \(G_J = \langle G_{J\cup \{i\}}~|~i\in I\setminus J\rangle\), then \(\Gamma_J\) is connected.
\end{prop}
\begin{proof}
If \(G_J = \langle G_{J\cup \{i\}}~|~i\in I\setminus J\rangle\), the coset geometry \(\Gamma(G_J,(G_{J\cup \{i\}})_{i\in I\setminus J})\) is connected. Proposition~\ref{chiral_final} implies that \(\varphi_J: \Gamma(G_J,(G_{J\cup \{i\}})_{i\in I}) \to \Gamma_J\) is surjective and therefore \(\Gamma_J\) is also connected.
\end{proof}
If we start from a \(C^+\)-group \((G^+,R)\) and construct the associated coset geometry with Construction~\ref{hyper}, the intersection condition ensures that every \(\Gamma_J\) is connected if \(|J| \le |I|-3\). It remains only to check what happens to the residues of rank \(2\) of \(\Gamma(G^+,(G_i^+)_{i\in I})\). Let us just state a small technical Lemma before doing so.
\begin{lemma}\label{lemma:chamber}
Let \((G^+,R)\) be a $C^+$-group and $\Gamma(G^+,(G_i^+)_{i\in I})$ be its associated geometry. Then \(C\) = \(\{G_0^+, G_1^+, ..., G_{n-1}^+\}\) is a chamber of  \(\Gamma\). Moreover, the set \(\{G_0^+, G_1^+, ..., \alpha_{i,j}^{-1}G_{j}^+,..., G_{n-1}^+\}\) is a \(j\)-adjacent chamber to \(C\) if \(i\) is in \(I\setminus \{j\}\).
\end{lemma}
\begin{proof}
By definition of \(G_k^+\), \( \alpha_{i,j}^{-1}\) is in \(G_k^+\) except if \( k = i\)  or \( j\). Therefore, since \(\alpha_{i,j}^{-1} G_j^+\) contains \(\alpha_{i,j}^{-1}\), it has a non empty intersection with every \(G_k^+\) with \(k \ne i,j\).
It just remains to check that \(\alpha_{i,j}^{-1} G_j^+ \cap G_i^+ \ne \emptyset\). For example, since \(\alpha_{i,j}^{-1} = \alpha_{j,i}\), we have that \(\alpha_{j,k} = \alpha_{j,i} \alpha_{i,k}\) is in \(\alpha_{i,j}^{-1} G_j^+\) for any \(k \ne i,j\).
By definition \(\alpha_{j,k}\) is also in \(G_i^+\) which concludes the proof.
\end{proof}

\begin{coro}
\label{firm}
Let \((G^+,R)\) be a $C^+$-group and $\Gamma:=\Gamma(G^+,(G_i^+)_{i\in I})$ be its associated geometry. Then \(\Gamma\) is firm.
\end{coro}
\begin{proof}
Every non maximal flag \(F\) of \(\Gamma\) is contained in a chamber \(C\). Since \(F\) is not maximal, there exists a type \(i\in I\) such that \(F\) has no element of type \(\{i\}\). Then, Lemma~\ref{lemma:chamber} gives us an \(i\)-adjacent chamber of \(C\), which also contains \(F\). 
\end{proof}

Let us now have a look at the residues of rank two of a chiral coset geometry \(\Gamma=\Gamma(G^+,(G_i^+)_{i\in I})\). For the sake of simplicity, let us first suppose that \(\Gamma\) is a rank \(3\) coset geometry.
In this case, \(\Gamma = \Gamma(G^+, (G_0^+,G_1^+,G_2^+))\) with \(G_0^+ = \langle \alpha_1^{-1}\alpha_2 \rangle\), \(G_1^+ = \langle \alpha_2 \rangle\) and \(G_2^+ =\langle \alpha_1 \rangle \).

Let us recall that \(G_i^+ \cap G_j^+ = \{e\} \) for all \(i \ne j \in \{0,1,2\}\)  by the intersection condition of \(C^+\)-groups. By Theorem~\ref{chiral_final} we have bijective homomorphisms \[\varphi_{\{i\}}: \Gamma(G_i^+, ({e},{e})) \mapsto \Gamma_{\{i\}}\] for any \(i =0,1,2\). Let us suppose, without loss of generality, that \(i = 0\).
Since \(\varphi_{\{0\}}\) is surjective, every element of type $j$ in \(\Gamma_{\{0\}}\) is of the form \(\alpha_{1,2}^k G_j^+\) for some natural number \(k \le o(\alpha_{1,2})\). Furthermore, we also know that the elements of \(\Gamma_{\{0\}}\) are two by two incident since \(\varphi_{\{0\}}\) is a homomorphism and the equivalence classes are trivial in this case.
More precisely, we know that \(\alpha_{1,2}^k G_1^+\) is incident to \(\alpha_{1,2}^k G_2^+\) for every \(k\).

Now, by Lemma~\ref{lemma:chamber}, we know that \((G_0^+, G_1^+,( \alpha_{1,2})^{-1} G_2^+)\) is a chamber that is 2-adjacent to \((G_0^+,G_1^+,G_2^+)\). This means that \(G_1^+\) is incident not only to \(G_2^+\) but also to  \((\alpha_{1,2})^{-1} G_2^+\). By using the action of \(\alpha_{1,2}\) on the incident pair \(\{G_1^+,(\alpha_{1,2})^{-1} G_2^+\}\), we obtain that \((\alpha_{1,2})^kG_1^+\) is incident to \((\alpha_{1,2})^{k-1} G_2^+\) for every \(k\). Putting this together with the incidence relation deduced from \(\varphi_{\{0\}}\) we get that \(\Gamma_{\{0\}}\) has an incidence graph which is at least a circuit of size \(2|G_0^+|\) . If \(\Gamma\) is thin, then the graph is fully determined. If \(\Gamma\) is not thin, there might be more incidence relations between the elements of \(\Gamma_{\{0\}}\).
In both cases, \(\Gamma_{\{0\}}\) is connected.

\begin{figure}
\begin{tikzpicture}[scale = 1.5]
\draw[black, thick](1,1.73) -- (-1,1.73)  node[pos = .45, black, anchor = south]{};
\draw[black, thick,dotted](-1,1.73) -- (-2,0)  node[pos = .45, black, anchor = south]{};
\draw[black, thick](-2,0) -- (-1,-1.73)  node[pos = .45, black, anchor = south]{};
\draw[black, thick,dotted](-1,-1.73) -- (1,-1.73)  node[pos = .45, black, anchor = south]{};
\draw[black, thick](1,-1.73) -- (2,0)  node[pos = .45, black, anchor = south]{};
\draw[black, thick,dotted](2,0) -- (1,1.73)  node[pos = .45, black, anchor = south]{};
\filldraw[black] (1,1.73) circle (1.5pt) node[anchor=south] {\(\alpha_{1,2}G_2^+\)};
\filldraw[black] (2,0) circle (1.5pt) node[anchor=west] {\((\alpha_{1,2})^2G_1^+\)};
\filldraw[black] (-2,0) circle (1.5pt) node[anchor=east] {\(G_2^+\)};
\filldraw[black] (-1,1.73) circle (1.5pt) node[anchor=south] {\(\alpha_{1,2}G_1^+\)};
\filldraw[black] (-1,-1.73) circle (1.5pt) node[anchor=north] {\(G_1+\)};
\filldraw[black] (1,-1.73) circle (1.5pt) node[anchor=north] {\((\alpha_{1,2})^2G_2^+\)};
\end{tikzpicture}
\centering
\caption{The incidence graph of a residue of rank \(2\) when \(|G_0^+| =3\).\label{fig:rank2residue}}
\end{figure}
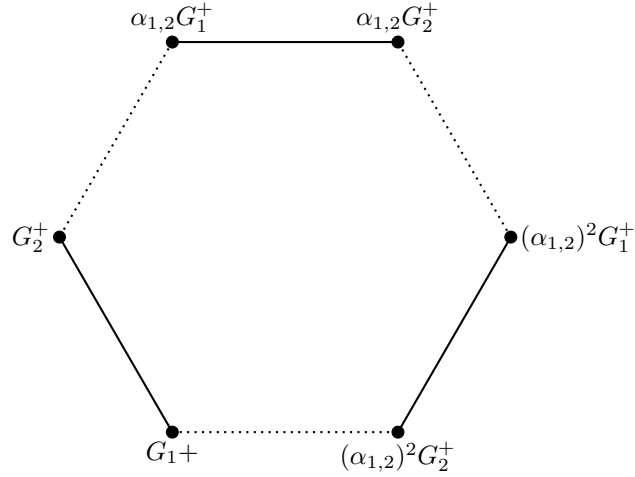
\begin{example}
Figure~\ref{fig:rank2residue} illustrates the reconstruction of the residue \(\Gamma_{\{0\}}\) using the method described in the above paragraph, supposing that \(\alpha_{1,2}\) is of order 3. The full lines represent the incidence relations deduced from the action of \(\alpha_{1,2}\) on the chamber \(\{G_0^+,G_1^+,G_2^+\}\) and the dotted lines represent the incidence relations deduced from the action of \(\alpha_{1,2}\) on the chamber \(\{G_0^+,G_1^+,(\alpha_{1,2})^{-1}G_2^+\}\)
\end{example}
\begin{prop}
\(\Gamma = \Gamma(G^+,(G_0^+,G_1^+,G_2^+))\) is residually connected and firm. 
\end{prop}
\begin{proof}
The surjectivity of \(\varphi_{\{i\}}\) for \(i = 0,1,2\) always holds as stated in an earlier remark, and we therefore need not to ask \(\Gamma\) to be chiral on this special case. The above construction then shows that every residue of rank \(2\) is connected. Moreover \(\Gamma\) itself is also connected since \(G^+= \langle G_i^+~|~i = 0,1,2\rangle\) by definition of \(C^+\)-groups.
Corollary~\ref{firm} says that \(\Gamma\) has to be firm.
\end{proof}

The construction described for rank \(2\) residues actually only depends on the surjectivity of \(\varphi_{\{i\}}\), that is guaranteed by Lemma~\ref{lemma:varphi}. We can thus extend this to coset geometries \(\Gamma(G^+,(G_i^+)_{i\in I})\) of any rank. Indeed, if we fix a \(J \subseteq I\) such that \(|J| = |I|-2\), then the coset geometry \(\Gamma(G_J,(G_{J\cup\{i\}})_{i\in I\setminus J})\) is generated by \(\alpha_{i,k}\) for the two elements \(i,k \in I\) that are not in \(J\) and the construction proceeds as for the rank \(3\) case. This permits to finish the proof of our main theorem.
\begin{proof}[Proof of Theorem~\ref{main}]
Putting everything together, the above construction yields that the rank \(2\) residues are connected and the intersection condition together with the surjectivity of \(\varphi_J\) for all \(J\subseteq I\) such that \(|J| \le |I| - 3 \) yields that the rank \(3\) or more residues are also connected. \(\Gamma\) itself is also connected by the intersection condition.
\end{proof}

\end{document}